\RequirePackage{tikz}
\documentclass[a4wide]{article}
\usepackage[utf8]{inputenc}
\usepackage{amsmath}
\usepackage{amssymb}
\usepackage{amsthm}
\usepackage{amsfonts}
\usepackage{url}
\usepackage[all,knot,poly]{xy}
\theoremstyle{definition}
\newtheorem{definition}{Definition}
\newtheorem{theorem}{Theorem}
\newtheorem{proposition}{Proposition}
\newtheorem{lemma}{Lemma}
\newtheorem{remark}{Remark}
\newtheorem{example}{Example}
\usepackage{tikz-cd}
\newcommand{\fl}{\rightarrow}
\def\Ab{\mbox{\bf Ab}}
\def\C{\mbox{$\cal C$}}

\def\ie{{\emph{i.e.~ }}}
\newcommand\Te{\mbox{$\cal T$}}
\newcommand\I{\overrightarrow{I}}
\newcommand\dipi[1]{\overrightarrow{\pi}_1(#1)}

\newcommand\Cr{\mathcal{C}}

\definecolor{mred}{rgb}{0.7,0.1,0.1}
\definecolor{mblue}{rgb}{0,0,0.8}
\definecolor{mgreen}{rgb}{0,0.6,0.3}

\newcommand\Fact[1]{{\mathcal{F}}#1}
\newcommand\RFact[1]{{\mathcal{R}}#1}
\newcommand\LFact[1]{{\mathcal{L}}#1}
\newcommand\map[3]{#1: #2 \longrightarrow #3}

\newcommand\ab[0]{\textbf{Ab}}

\newcommand\gp[0]{\textbf{V}}

\newcommand\natsys[2]{NatSys(#1,#2)}

\newcommand\Real[0]{\mathbb{R}}

\newcommand\N[0]{\mathbb{N}}
\newcommand\D[0]{\mathcal{D}}

\newcommand\V[0]{\mathcal{V}}

\def\TM{\mbox{$\mathbb T$}}
\def\SM{\mbox{$\mathbb M$}}

\newcommand\functor[1][l]{\csname#1functor\endcsname}
\newcommand\lfunctor[3]{%
  \setbox0=\hbox{$#2$}%
  \kern\wd0%
  \ensurestackMath{\Centerstack[c]{#1\\ \mathllap{#2\;\,}\mathclap{\DownArrow}\\#3}}%
}
\newcommand\rfunctor[3]{%
  \setbox0=\hbox{$#2$}%
  \ensurestackMath{\Centerstack[c]{#1\\\mathclap{\DownArrow}\mathrlap{\,\;#2}\\#3}}%
  \kern\wd0%
}

\newcommand\DownArrow{\rotatebox[origin=c]{-90}{$\longrightarrow$\,}}

\begin{document}

\title{A semi-abelian approach to directed homology}
\author{Eric Goubault \\
{LIX}, {CNRS, Ecole Polytechnique,} \\{Institut Polytechnique de Paris}, \\ {{Palaiseau}, {91128 Cedex}, {France}}}



\maketitle

\begin{abstract}
We develop a homology theory for directed spaces, based on the semi-abelian category of (non-unital) associative algebras. The major ingredient is a simplicial algebra constructed from convolution algebras of certain trace categories of a directed space. We show that this directed homology HA is invariant under directed homeomorphisms, and is computable as a simple algebra quotient for $HA_1$. We also show that the algebra structure for $HA_n$, $n\geq 2$ is degenerate, through a Eckmann-Hilton argument. We hint at some relationships between this homology theory and natural homology, another homology theory designed for directed spaces. Finally we pave the way towards some interesting long exact sequences.
\end{abstract}

\section{Introduction}

The purpose of this article is to present directed homologies, closer to ordinary homology theories, using classical structures coming from associative algebra. The hope is to use well-established methods from non-abelian homology, such as the semi-abelian framework \cite{VanderLinden} we consider here, and which would be easier to manipulate and to compute than 
the ones developped in e.g. \cite{Dubut,cameron}, which relied on involved categorical structures (natural systems, with composition pairing). 
The long-term objective is to obtain practical, computable invariants for directed topology, that could even be implemented using existing tools from computer algebra. Another long-term objective is to relate directed homology to other theories, persistence, semi-abelian categories and representation theory of algebras.

The first intuition underlying this work is that we can make directed homological calculations in associative algebras, in place of modules. An associative algebra encodes one more operation than the addition and external multiplication operations, which are ways in ordinary algebraic topology, to represent formal sums of cells, and "count" cells and holes. This extra operation is internal multiplication, that models composition of cells in this work. Classical homology theories with coefficients in abelian groups equate concatenation with addition, whereas in directed topology, we believe it to be of primary importance to separate the two operations. 

Another intuition comes from \cite{Dubut,cameron}. 
There are direct relationships between what we are doing here with algebras and (bi)modules over algebras with natural systems (with compositional pairing) as developed in \cite{Dubut,cameron}. Informally, a natural system as in \cite{Dubut,cameron} is a generalized bimodule, with actions on the left and on the right of one dimensional paths, and composition pairing is a mirror of the algebra operation on the underlying paths. 


The homology theory we develop is quite simple and well-behaved, at the expense of having to apply this to constructions (glueing etc.) on directed spaces in a more contrived way than one would like them to be, because of functoriality issues. A way to cope with this would be to go from algebras to algebroids ("algebras with many objects"), this will be developed elsewhere since this requires a whole new machinery, which seems to go beyond the semi-abelian framework.

\paragraph{Contents}

In order to keep the article self-contained, we recap the notions we need, about directed spaces in Section 
\ref{sec:dspace}, about the semi-abelian category of associative algebras in Sections 
\ref{sec:backgroundalgebra} and 
\ref{sec:catalg}, and about the central construct of convolution algebras in Section 
\ref{sec:convalg}. 

The development of our directed homology theory starts at Section 
\ref{sec:simpalg} with the construction of a simplicial object in the category of (non-unital) associative algebras. Algebras in all dimensions $i \geq 1$ are built from $i$-traces, a form of directed hyper-surface of dimension $i$, together with a concatenation operation, which is non-commutative since it respects the underlying directed structure. Non concatenable traces have zero product. 

Using the semi-abelian framework \cite{VanderLinden}, we construct a homology theory $HA$ from the simplicial algebra construction of Section 
\ref{sec:diralg}. We derive its first elementary properties in Section 
\ref{sec:elemprop}. First we show that indeed, this homology theory is invariant under directed homeomorphisms, then we show that the underlying modules of $HA$ of a directed space are the same modules as the ones obtained from natural homology, for each pairs of initial and final points. The computation of $HA_1(X)$ is a simple one, that we fully spell out. Then applying the semi-abelian framework shows that the higher homology algebras of a directed space $X$, $HA_n(X)$, $n\geq 2$, are mere modules: the algebra operation collapses to the null multiplication\footnote{We are in the category of non-unital associative algebras, this does not imply that the algebra is the zero algebra.}. We finally hint at deeper relations with the approach to natural homology taken in \cite{cameron}, with the composition pairing  operation defining an internal multiplication of algebras. 

We illustrate some simple computations of directed homology algebras on the geometric realization of simple precubical sets in Section  
\ref{sec:examples}. In the restricted cases we exhibit, we can compute interesting sub-algebras of these directed homology algebras, which classify the traces up to homology between any pair of vertices of the precubical set. 

Finally, we discuss the existence of interesting long exact sequences in this directed homology theory in Section 
\ref{sec:exactseq}, applying again the semi-abelian framework. We exhibit a long exact sequence relating the homology algebras of a coproduct of two spaces with that of each of the two spaces, and discuss the potential generalizations to non-disjoint unions of spaces, paving the way for a general Mayer-Vietoris like long exact sequence. 

\section{Directed spaces}

\label{sec:dspace}

The context of a directed space space was introduced in \cite{grandisbook}.    

Let $I=[0,1]$ denote the unit
segment with the topology inherited from $\Real$. 

\begin{definition}[\cite{grandisbook}]
A directed topological space, or a d-space is a pair  $(X,dX)$ consisting of a topological space $X$ equipped with
a subset $dX\subset X^I$ of continuous paths $p:I \rightarrow X$, called directed paths or
d-paths, satisfying three axioms: 
\begin{itemize}
\item every constant map $I\rightarrow X$ is directed;
\item $dX$ is closed under composition with continuous non-decreasing maps  $I\to I$;
\item $dX$ is closed under concatenation.
\end{itemize}
\end{definition}

We shall abbreviate the notation $(X,dX)$ to $X$. 

Note that for a d-space $X$, the space of d-paths $dX\subset X^I$ is a topological space, it is equipped with the compact-open topology. 

A map $f: X\to Y$ between d-spaces $(X, dX)$ and $(Y, dY)$ is said to be {\it a d-map} if it is continuous and for any d-path $p\in dX$ 
the composition $f\circ p:I\to Y$ belongs to $dY$. In other words we require that $f$ preserves d-paths. We write $df~: dX \rightarrow dY$ for the induced map between directed paths
spaces. We denote by ${\cal D}$ the category of d-spaces. 

An isomorphism (also called dihomeomorphism) between directed spaces is a homeomorphism which is a directed map, and whose inverse is also a directed map. One of the ultimate goals of directed topology is to classify directed spaces up to dihomeomorphisms. 

Classical examples of d-spaces arise as geometric realizations of precubical sets, as found in e.g. the semantics of concurrent and distributed systems \cite{thebook}. We will use some of these classical example to explain the directed homology theory we are designing in this article, hence we need to recap some of the basic definition on precubical sets, as well as their relation to d-spaces. 


\begin{definition}
A precubical set $C$ is a sequence of disjoint sets $(C_n)_{n \geq 0}$ with a collection of face maps 
$$d^{\epsilon}_i : \ C_n \rightarrow C_{n-1}$$ 
for $n > 0$, $\epsilon \in \{0,1\}$, $i \in \{1,\ldots,n\}$ such that $d^\epsilon_i d^\eta_j = d^\eta_{j-1}d^\epsilon_i$ for all $\epsilon, \eta \in \{0, 1\}$ and $i < j$.
\end{definition}

Precubical sets are easily seen to form a presheaf category on a site $\Box$, whose objects are $[n]$, $n\in \N$ and whose morphisms are $\delta^\epsilon_i: \ [n-1] \rightarrow [n]$ for $n >0$, $\epsilon \in \{0,1\}$, $i \in \{1,\ldots,n\}$, such that $\delta^\eta_j \delta^\epsilon_i  = \delta^\epsilon_i \delta^\eta_{j-1}$, for 
$i < j$. 

Let $\I=[0,1]$ with as directed paths all increasing paths in the usual ordering inherited from real numbers, and $\I^n$ the $n$th power of $\I$ (for $n > 0$) as a directed space: as ${\cal D}$ is complete (as well as co-complete), this is well defined. $\I^n$ has as directed paths all continuous paths that are increasing on each coordinate. 

\begin{definition}
The geometric realization $\mid C\mid $ of a precubical set $C$ is the Yoneda extension of the following functor $F: \ \Box \rightarrow {\cal D}$: 
\begin{itemize}
    \item $F([n])=\I^n$
    \item $F(\delta^\epsilon_i)$ is the map from $\I^{n-1}$ to $\I^n$ which sends $(x_1,\ldots,x_{n-1})\in \I^{n-1}$ to $(x_1,\ldots,x_{i},\epsilon,x_{i+1},\ldots,x_{n-1})$. 
\end{itemize}
\end{definition}

For more details, see e.g. 
\cite{Ziemanski2}.


\section{Associative algebras}



\label{sec:backgroundalgebra}
Let $R$ be a commutative ring.  

\begin{definition}
An (non-unital) associative algebra on $R$, or $R$-algebra $A=(A,+,.,\times)$ is a $R$-module $(A,+,.)$, with external multiplication by elements of the ring $R$ denoted by ., that has an internal semigroup operation, which is an associative operation ("multiplication" or "internal multiplication") $\times: A \times A \rightarrow A$ that is bilinear. We denote by $0$ the neutral elements for $+$ (which we use also for denoting the 0 of the ring $R$). 
\end{definition}


\begin{definition}
Let $A$ and $B$ be two $R$-algebras. A morphism $f: A \rightarrow B$ of $R$-algebras is a linear map from $A$ to $B$ seen as $R$-modules, such that it commutes with the internal multiplication: 
$$f(a\times a')=f(a)\times f(a')$$
\end{definition}

We call $Alg$ the category of $R$-algebras .

A $R$-submodule $B$ of a $R$-algebra $A$ is a $R$-subalgebra of $A$ if $b_1\times b_2 \in B$ for all $b_1, b_2 \in B$. 
A $R$-submodule $I$ of a $R$-algebra $A$ is a right ideal of $A$ (resp. left ideal of $A$) if $x\times a \in I$ (or $a\times x \in I$, respectively) for all $x \in I$ and $a \in A$. A two-sided ideal of $A$ (or simply an ideal of $A$) is a both a left ideal and a right ideal of $A$.

If $I$ is a two-sided ideal of a $R$-algebra $A$, then the quotient $R$-module $A/I$ has a unique $R$-algebra structure, which is the quotient of the algebra $A$ by $I$, such that the canonical surjective linear map $\pi : A \rightarrow A/I$, $a \rightarrow a = a + I$, becomes a $R$-algebra homomorphism.

Some particular elements will play an important role in our constructions, that are also central in representation theory \cite{assocalg}: 
let $A$ be a $R$-algebra, an element $e \in A$ is called an idempotent if $e^2 = e$.

The idempotent $e$ is said to be central if $a\times e = e\times a$ for all $a \in A$. The idempotents $e_1, e_2 \in A$ are called orthogonal if $e_1 \times e_2 = e_2\times e_1 = 0$. The idempotent $e$ is said to be primitive if $e$ cannot be written as a sum $e = e_1 + e_2$, where $e_1$ and $e_2$ are nonzero orthogonal idempotents of $A$.



\section{Categorical and homological properties of algebras}

\label{sec:catalg}

Since associative algebras are models of an algebraic theory, the category of associative algebras is complete and cocomplete by classical results from the theory of Lawvere theories, see e.g. \cite{lawveretheories}. 


The category of associative algebras is not Abelian. Although $Alg$ has a zero object (the associative algebra comprising only 0), and has coproducts and products, these do not coincide, therefore $Alg$ is not even additive. 

Indeed, the coproduct of $A_1 \coprod A_2$ of two $R$-algebras $(A_1,+_1,._1,\times_1)$ and $(A_2,+_2,._2,\times_2)$ is $A=(A,+, ., \times)$ where $A$ is the coproduct, as $R$-modules, of all $$A_{\epsilon_1}\otimes A_{\epsilon_2} \otimes \ldots \otimes A_{\epsilon_n}$$ \noindent where $n\geq 0$, $\epsilon_i\in \{1,2\}$ for all $i=1,\ldots, n$ and $\epsilon_j \neq \epsilon_{j+1}$ for all $j=1,\ldots,n-1$, with the following internal multiplication. If $(a_1,\ldots,a_n)\in A_{\epsilon_1}\otimes A_{\epsilon_2} \otimes \ldots \otimes A_{\epsilon_n}$ and $(b_1,\ldots,b_m)\in A_{\delta_1}\otimes A_{\delta_2} \otimes \ldots \otimes A_{\delta_n}$ then: 
    $$(a_1 \otimes \ldots \otimes a_n)\times (b_1 \otimes \ldots \otimes b_m)=\left\{\begin{array}{ll}
    a_1\otimes \ldots \otimes a_n \otimes b_1 \otimes \ldots \otimes b_m & \mbox{if $\epsilon_n\neq \delta_1$} \\
    a_1\otimes \ldots \otimes (a_n \times b_1) \otimes \ldots \otimes b_m & \mbox{otherwise} 
    \end{array}\right.$$
    
    The canonical embeddings are the maps $in_1: \ A_1 \rightarrow A_1 \coprod A_2$ and $in_2: \ A_2 \rightarrow A_1 \coprod A_2$ which are the obvious identity maps, embedding $A_1$ and $A_2$ into the coproduct of $A_1$ and $A_2$ as $R$-algebras. 
    
    Whereas 
    the product $A_1 \times A_2$ of $(A_1,+_1,._1,\times_1)$ and $(A_2,+_2,._2,\times_2)$, two $R$-algebras, is $A=(A,+, ., \times)$ where $A$ has as underlying $R$-module the product of the $R$-module $A_1$ with the $R$-module $A_2$, that is, elements of which are $(a_1,a_2)$, with $a_1\in A_1$ and $a_2 \in A_2$, and with algebra multiplication: 
    $$
    (a_1,a_2)\times (a'_1,a'_2)=(a_1\times_1 a'_1,a_2\times_2 a'_2)
    $$
    
    Even though $Alg$ is not Abelian, it is possible to do homological algebra within $Alg$ since it is a semi-Abelian category \cite{VanderLinden}. 
This is due to the fact again that $Alg$ forms an algebraic variety in the sense of Lawvere, and in order for algebraic varieties to be semi-abelian, it is enough that they have a zero object and that they are protomodular in the sense of Bourn. By algebraic reasoning, an algebraic variety is protomodular in particular when the algebraic variety includes a group operation \cite{Borceux}, which is obviously the case for $Alg$.

In what follows, we will need to use classical kernel and co-kernel constructions to define a directed homology in the semi-Abelian category $Alg$, that we are now describing. 

     The equalizer of pair of morphisms in $Alg$, $f, \ g: \ A \rightarrow B$ is 
    $$Ker(f,g)=\{a \in A \mid f(a)=g(a) \}
    $$
    \noindent with the inclusion morphism $ker(f,g): \ Ker(f,g) \rightarrow A$. We write $Ker \ f$ for $Ker(f,0)$ and $ker \ f: \ Ker \ f \rightarrow A$ is the corresponding inclusion morphism. $Ker(f)$ is indeed a two-sided ideal of $A$ whereas $Ker(f,g)$ may not be. 
    
    We call $Ker(f)$, $Ker(f,0)$, the kernel of $f$.

Let again $f, \ g: \ A \rightarrow B$ where $A$ and $B$ are $R$-algebras. Then,
$$
coKer(f,g)=B/I(f,g)
$$
\noindent  where $I(f,g)$ is the two-sided ideal of $B$ generated by $Im \ (f-g)$, and $coker(f,g): \ B\rightarrow coKer(f,g)$ is the corresponding epimorphism (i.e. a surjective morphism) of $R$-algebras. 

We write $coKer \ f$ for $coKer(f,0)$ and $coker \ f: \ B \rightarrow coKer \ f$.
    
Now, if $f$ is proper in the sense that its image is normal, i.e. its image is a two-sided ideal in $B$, $coKer(f)=B/Im(f)$. 





\section{Convolution algebras}
\label{sec:convalg}

\begin{definition}
\label{def:categoryalgebra}
The category algebra or convolution algebra 
$R[\C]$ of $\C$ over $R$ is the $R$-algebra
whose underlying $R-module$ is the free module 
$R[\C_1]$ over the set of morphisms of $\C$ and 
whose product operation is defined on basis-elements 
$f$, $g\in \C_1 \subseteq R[\C]$
to be their composition if they are composable and zero otherwise:
$$f\times g=
\left\{\begin{array}{ll}
g\circ f & \mbox{if composable} \\
0 & \mbox{otherwise} 
\end{array}
\right.$$
\end{definition}

When $\C$ is a poset, $R[\C]$ is the incidence algebra of the poset. When $\C$ is a quiver, $R[\C]$ is known as the path algebra of $\C$. 
When $\C$ is a category, $R[\C]$ is a quotient of the path algebra of its underlying quiver by the ideal generated by relations $f\times g-g\circ f=0$. Note that when $\C$ is a groupoid, $R[\C]$ is a star-algebra (an algebra equipped with an anti-involution).  

\begin{remark}
\label{rem:decomppath}
In all cases, path algebras $R[Q]$ are such that they are the coproduct, as $R$-modules, of the $R[Q](a,b)=e_aR[Q]e_b$ for all $a$, $b\in Q_0$. 
\end{remark}

We give below a number of classical examples of quiver and poset algebras. 

\begin{example}[directed $S^1$: $dS^1$]
\label{ex:dirS1}
Consider the simple loop graph: 
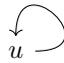
\[
\begin{tikzcd}
u \arrow[out=0,in=90,loop]
\end{tikzcd}
\]
Its path algebra is easily seen to be the algebra of polynomials in one indeterminate $R[t]$. Indeed, the basis of the $R$-module of dipaths is in bijection with $\{1, t, t^2, \ldots\}$ (where $t^i$ denotes the unique path of length $i$), and the algebra multiplication adds up lengths of dipaths. 
\end{example}

\begin{example}[see e.g. \cite{assocalg})]
\label{ex:kronecker}
Consider another graph version of a directed circle: 
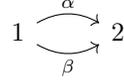
\[
\begin{tikzcd}
1 \arrow[r,bend left,"\alpha"] \arrow[r,bend right,swap,"\beta"] & 2
\end{tikzcd}
\]

In that case, the corresponding path algebra is 
the following matrix algebra: 
$$
\begin{pmatrix}
R & 0 \\
R^2 & R
\end{pmatrix}
$$
\noindent known as the Kronecker algebra. Elements of this Kronecker algebra are of the form: 
$$
\begin{pmatrix}
a & 0 \\
(b,c) & d
\end{pmatrix}
$$
\noindent with the "obvious" addition and external multiplication, and as internal (algebra) multiplication, the "obvious one" as well: 
$$
\begin{pmatrix}
a & 0 \\
(b,c) & d
\end{pmatrix}
\times
\begin{pmatrix}
a' & 0 \\
(b',c') & d'
\end{pmatrix}
=
\begin{pmatrix}
a'a & 0 \\
(a'b+b'd,a'c+c'd) & d'd
\end{pmatrix}
$$

\end{example}

\begin{example}[Empty square]
\label{ex:emptysquare}
We apply Lemma 1.12 of \cite{assocalg} to the following graph: 
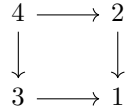
\[\begin{tikzcd}
  4 \arrow[r] \arrow[d]
    & 2 \arrow[d] \\
  3 \arrow[r]
& 1 \end{tikzcd}
\]
We get the following path algebra: 
$$
\begin{pmatrix}
R & 0 & 0 & 0 \\
R & R & 0 & 0 \\
R & 0 & R & 0 \\
R^2 & R & R & R 
\end{pmatrix}
$$
Entries $(i,j)$ are $R^k$ where $k$ is the number of paths from vertex $i$ to vertex $j$. 
\end{example}


\begin{remark}
\label{rem:notfunctorial}
Finally, note that the category algebra is not a functorial construction. This is because there is no reason, given a functor $F$ from $\C$ to $\D$, its linearization (the obvious linear map induced by $F$ on the underlying $R$-module of $R[\C]$) preserves internal multiplication. 

To see this, consider the two object category $2$, objects are called $\alpha$ and $\beta$, with only identities as morphisms, and the one object category $1$, with only object $\gamma$, and only morphism being the identity on $\gamma$. Consider now the functor $F: \ 2 \rightarrow 1$ which is the unique possible map from $2$ to $1$. The algebra $R[2]$ is $R^2$ as a $R$-module, with multiplication $(k,l) \times (k',l')=(kk',ll')$. The algebra $R[1]$ is, as a $R$-module, $R$, with algebra multiplication $m \times m'=mm'$. Now, the linearization of $F$ maps $(k,l)$ to $k+l$, but $F((k,l)\times (k',l'))=kk'+ll'$ which is not in general equal to $F(k,l)\times F(k',l')=(k+l)(k'+l')$.

In fact, only injective-on-objects functors are mapped naturally onto algebra morphisms. Indeed, suppose $F$ is a functor from $\C$ to $\D$, that is injective on objects, and consider its linearization, that we still denote by ${F}$, from the free $R$-module enriched category $\C$ to the free $R$-module enriched category $\D$. We argue that ${F}$ induces an algebra map from $R[\C]$ to $R[\D]$. We compute: 
$$\begin{array}{rcl}
F\left(\left(\sum\limits_{i=1}^n r_i f_i\right)\times \left(\sum\limits_{j=1}^m s_i g_j\right)\right) & = & F\left(\sum\limits_{i=1,\ldots,n; j=1,\ldots,m} 
r_i s_j (g_j \circ f_i)\right)
\end{array}$$
\noindent where $f_i$ and $g_j$ are morphisms of $\C$ and $r_i$ and $s_j$ are elements of $R$. Let $S$ be the subset of $\{1,\ldots,n\}\times \{1,\ldots,m\}$ of composable pairs $(f_i,g_j)$, i.e. $(i,j)\in S$ if and only if $g_j$ composes with $f_i$. We can now write: 
$$\begin{array}{rcl}
F\left(\left(\sum\limits_{i=1}^n r_i f_i\right)\times \left(\sum\limits_{j=1}^m s_i g_j\right)\right) & = & F\left(\sum\limits_{(i,j)\in S} 
r_i s_j (g_j \circ f_i)\right) \\
 & = & \sum\limits_{(i,j)\in S} r_i s_j (F(g_j) \circ F(f_i))
\end{array}$$
Now, as $F$ is injective on objects, $F(g_j)$ composes with $F(f_i)$ if and only if $g_j$ composes with $f_i$, that is, if and only if $(i,j) \in S$. Therefore: 
$$\begin{array}{rcl}
F\left(\left(\sum\limits_{i=1}^n r_i f_i\right)\times \left(\sum\limits_{j=1}^m s_i g_j\right)\right) & = & \sum\limits_{(i,j)\in S} r_i s_j (F(g_j) \circ F(f_i)) \\
& = & \left(\sum\limits_{i=1}^n r_i F(f_i)\right)\times \left(\sum\limits_{j=1}^m s_i F(g_j)\right)\\
& = & F\left(\sum\limits_{i=1}^n r_i f_i\right)\times F\left(\sum\limits_{j=1}^m s_i g_j\right)
\end{array}$$
\end{remark}

\section{The simplicial algebra of a directed space}

\label{sec:simpalg}

Let $\I$ be the directed interval as in Section \ref{sec:dspace}, and $I$ be the directed space $[0,1]$ where all paths are directed. 

In next section, we are going to define a directed homology theory based on the construction of a simplicial object in the semi-abelian category $Alg$ we are introducing here. 





We recall the following: 
the standard simplex of dimension $n$ is $$
\Delta_n=\left\{(t_0,\ldots,t_n) \mid \forall i\in \{0,\ldots,n\}, \ t_i \geq 0 \mbox{ and } \sum\limits_{j=0}^n t_j=1\right\}
$$
For $n \in \N$, $n
\geq 1$ and $0 \leq k \leq n$, the $k$th ($n-1$)-face (inclusion) of the topological $n$-simplex is the subspace inclusion
$$\delta_k: \ \Delta_{n-1} \rightarrow \Delta_n$$
induced by the inclusion
$$(t_0,\ldots,t_{n-1}) \rightarrow (t_0,\ldots,t_{k-1},0,t_k,
\ldots, t_{n-1})$$
For $n \in \N$ and $0\leq k < n$, the 
$k$th degenerate 
$n$-simplex is the surjective map
$$
\sigma_k: \ \Delta_n \rightarrow \Delta_{n-1}$$
\noindent induced by the surjection: 
$$(t_0,\ldots,t_n)\rightarrow (t_0,\ldots,t_{k}+t_{k+1},\ldots,t_n)$$

Now we define traces of all dimensions in a directed space, and give them the structure of a simplicial set, then of a simplicial algebra. 

\begin{definition}
Let $X$ be a directed space, $i\geq 1$ an integer. 
We call $p$ an $i$-path, or path of dimension $i$ of $X$, any 
 directed map 
$$p: \I \times \Delta_{i-1} \rightarrow X$$ 
An $i$-trace from $p_s$ to $p_t$ is a class of $i$-path:
\begin{itemize}
\item modulo continuous increasing reparameterization in the first coordinate, 
\item with $p(0,t_0,\ldots,t_{i-1})$ not depending on $t_0,\ldots,t_{i-1}$ and equal to $p_s$ 
\item and $p(1,t_0,\ldots,t_{i-1})$ constant as well, equal to $p_t$
\end{itemize}
We write $T_i(X)$ for the set of $i$-traces in $X$. We write $T_i(X)(a,b)$, $a$, $b \in X$, for the subset of $T_i(X)$ made of $i$-traces from $a$ to $b$. 
\end{definition}

\begin{definition}
\label{def:boundaries}
Let $X$ be a directed space. We define:
\begin{itemize}
\item Maps $d_{j}$, $j=0,\ldots,i$ acting on $(i+1)$-paths 
$p: \I \times \Delta_i \rightarrow X$, $i\geq 1$: 
$$d_j(p)=p\circ (Id,\delta_j)$$
\item Maps $s_k$, $k=0,\ldots,i-1$ acting on $i$-paths 
$p: \I \times \Delta_{i-1} \rightarrow X$, $i\geq 1$: 
$$s_k(p)=p\circ (Id,\sigma_k)$$
\end{itemize}
\end{definition}

\begin{lemma}
\label{lem:boundariesstartend}
Maps $d_j$ defined in Definition \ref{def:boundaries} induce maps from $T_{i+1}(X)$ to $T_i(X)$. Similarly, maps $s_j$ induce maps from $T_i(X)$ to $T_{i+1}(X)$. 

Furthermore, these maps restrict to maps from $T_{i+1}(X)(a,b)$ to $T_i(X)(a,b)$ for any $a$, $b \in X$ (respectively, from $T_i(X)(a,b)$ to $T_{i+1}(X)(a,b)$). 
\end{lemma}

\begin{proof}
Maps $d_j$ and $s_k$ of Definition \ref{def:boundaries} are compatible with  reparametrization: two $(i+1)$-paths $p$ and $q$ equivalent modulo reparametrization (in its first parameter) give equivalent boundaries; similarly for degeneracy maps. 

The last point we have to check is that the image of the boundary operators (resp. degeneracy operators) give $i$-traces from $(i+1)$-traces (resp. $(i+1)$-traces from $i$-traces). It is due to the fact that, indeed, 
%
if $p$ is a $(i+1)$-trace from $p_s$ to $p_t$, $d_j(p)(0,s_0,\ldots,s_{i-1})=d_j(p)(0,d_j(s_0,\ldots,s_{i-1}))=p_s$, the other computations being similar. This also proves that boundary (resp. degeneracy) operators restrict to maps from $T_{i+1}(X)(p_s,p_t)$ to $T_i(X)(p_s,p_t)$ (resp. from $T_i(X)(p_s,p_t)$ to $T_{i+1}(p_s,p_t)$). 
\end{proof}

\begin{lemma}
\label{lem:singsimpset}
Let $X$ be a directed space. The boundary and degeneracy operators of Definition \ref{def:boundaries} give the sequence $ST(X)=(T_{i+1}(X))_{i\geq 0}$ the structure of a simplicial set. 

This restricts to a simplicial set $ST(X)(a,b)$ of traces from $a$ to $b$ in $X$, which is the singular simplicial set of the space of 1-traces (paths modulo continuous and increasing reparametrization) from $a$ to $b$, with the compact-open topology. 
\end{lemma}

\begin{proof}
The simplicial relations are direct consequences of the simplicial relations on $\delta_j$ and $\sigma_k$. The fact that $ST(X)(a,b)$ inherits the simplicial structure comes from Lemma \ref{lem:boundariesstartend}. The fact that this simplicial set is the singular simplicial set of the space of 1-traces from $a$ to $b$ is a direct application of curryfication. 
\end{proof}

We are now going to define higher trace categories from $T_i(X)$, for any directed space $X$, in order to recap the construction of natural homology \cite{Dubut}, and construct our new directed homology, in next section: 

\begin{definition}
\label{def:boundaryoperators}
Let $X$ be a directed space. We define a sequence of categories $\Te_i(X)$, $i=1,\ldots$ as follows: 
\begin{itemize}
    \item $\Te_i(X)$ has as objects, all points of $X$ 
    \item and as morphisms from $s$ to $t$, all $i$-traces from $s$ to $t$
     \item the composition $q\circ p$ is the concatenation $p*q$ of such $i$-traces, which is associative since this is taken modulo reparameterization: 
     $$
     p * q(s,t_0\ldots, t_i)=\left\{\begin{array}{ll}
     p(2s,t_0,\ldots,t_i) & \mbox{if $s\leq \frac{1}{2}$} \\
     q(2s-1,t_0,\ldots,t_i) & \mbox{if $\frac{1}{2} \leq s \leq 1$}
     \end{array}\right.
     $$
\end{itemize}
\end{definition}

Note that this allows us to reformulate the natural homology of directed spaces introduced in \cite{Dubut}: 

\begin{lemma}
\label{lem:naturalhomology}
The natural homology of \cite{Dubut} is the collection of functors $HN_i(X): {\cal F}\Te_1(X)\rightarrow Ab$, $i \geq 1$, from the factorization category of $\Te_1(X)$ to the category of abelian groups with: 
\begin{itemize}
\item $HN_i(X)(p)=H_i(ST(X)(a,b))$, where $p \in \Te_1(X)$ is a 1-trace from $a$ to $b \in X$, and $H_i$ is the standard homology of the unnormalized complex generated by the simplicial set $ST(X)(a,b)$,
\item $HN_i(X)(\langle u,v\rangle)(p)=[v \circ p \circ u]$ where $u$ is a 1-trace from $a'$ to $a$, $v$ a 1-trace from $b$ to $b'$, and $[x]$ denotes here the class in $H_i(X)(ST(X)(a',b')$ of $x$. 
\end{itemize}
\end{lemma}

\begin{proof}
Direct consequence of Lemma \ref{lem:singsimpset}.
\end{proof}

We will get back to relations between the directed homology algebras we are introducing in the next section, with natural homology, in Section \ref{sec:elemprop}.


\begin{definition}
\label{def:RX}
For $i\geq 1$, we define the $R$-algebra $R_i[X]$ to be the category algebra of ${\cal T}_i(X)$. 
\end{definition}

The elements $e^i_a$ of $R_i[X]$ which correspond in this algebra to the constant $i$-traces on $a$ are the idempotents in $R_i[X]$. 

\begin{lemma}
\label{lem:Rmodoutofalg}
For all $a$, $b \in X$, $e^i_a R_i[X] e^i_b$ is the free $R$-module generated by $T_i(X)(a,b)$. 
\end{lemma}

We write $R_i[X](a,b)$ for $e^i_a R_i[X] e^i_b$. 

\begin{proof}
By Definitions \ref{def:categoryalgebra} and \ref{def:RX}, $R_i[X]$ is, as a $R$-module, the coproduct of all free $R$-modules on $T_i(X)(a,b)$, with $a$, $b$ varying over $X$. The algebra operation with idempotents $e^i_a$ and $e^i_b$ does the rest. 
\end{proof}

\begin{theorem}
\label{lem:simplicialalgebra}
The collection of $R$-algebras $R[X]=(R_{i+1}[X])_{i \geq 0}$ can be given the structure of a Kan simplicial object in $Alg$. 

Furthermore 
this simplicial structure restricts to 
$$R[X](a,b)=(R_{i+1}[X](a,b))_{i \geq 0}$$ 
which is the free (Kan) simplicial $R$-module on $ST(X)(a,b)$. 
\end{theorem}

\begin{proof}
Indeed, we can extend by linearity the $d_j$ and $s_k$ operators from $T_{i+1}(X)$ to $T_i(X)$, resp. from $T_i(X)$ to $T_{i+1}(X)$, to the underlying $R$-module of the algebra $R[X]$. 

Now, we have to prove that all these maps are algebra maps. We compute, for any $(i+1)$-trace $p$ from $r$ to $s$, and $q$ $(i+1)$-trace from $s$ to $t$: 
$$
\begin{array}{rcl}
d_j(p\times q)(s,t_0,\ldots,t_i) & = & p\times q(s,d_j(t_0,\ldots,t_{i})) \\
& = & 
\left\{\begin{array}{cc}
p(2s,d_j(t_0,\ldots,t_{i})) & \mbox{if $s\leq \frac{1}{2}$} \\
     q(2s-1,d_j(t_{0},\ldots,t_{i})) & \mbox{if $\frac{1}{2} \leq s \leq 1$}
\end{array}\right.
\end{array}
$$
Hence: 
\begin{equation}
d_j(p\times q)(s,t_0,\ldots,t_i) =  (d_j(p)\times d_j(q))(s,t_0,\ldots,t_i)
\label{eq:bound}
\end{equation}
\noindent and similarly for the degeneracy operators. 
This holds as well when $p\times q=0$: $d_j(p\times q)=0=d_j(p)\times d_j(q)$, and similarly for degeneracy operators $s_k$. 
The $d_j$ and $s_k$ operators being linear, this equation holds for all elements of $R[X]$. 

The free simplicial $R$-module on $ST(X)(a,b)$, for all $a$, $b\in X$, adjoint to the forgetful functor from simplicial $R$-modules to simplicial sets, is given by the collection of free $R$-modules generated by each $T_i(X)(a,b)$, $i\geq 1$, together with linear extensions of degeneracy and boundary maps in $ST(X)(a,b)$. The fact that these free $R$-modules are equal to $R_{i}[X](a,b)$ for all $a$, $b \in X$, is a direct consequence of Lemma \ref{lem:Rmodoutofalg}.

Finally, being Kan is already the case for simplicial groups \cite{May}, and  simplicial algebras are in particular simplicial (abelian) groups. 
\end{proof}

Note that we have even an augmentation in the underlying semi-simplicial algebra structure $R[X]$, with the augmentation given by   
 the following $R$-algebra, that we call $R_0[X]$ to keep our indexes consistent. Its underlying $R$-module is generated by pairs $(x,y)$ of points in $X$, such that there exists a directed path from $x$ to $y$ in $X$.  The algebra operation is, for $(x,y)$ and $(z,t) \in R_0[X]$:
$$(x,y) \times (z,t)=\left\{\begin{array}{ll}
(x,t) & \mbox{if $y=z$}\\
0 & \mbox{otherwise}
\end{array}\right.
$$
\noindent and extended by linearity. 
This augmentation map is $\partial_0: \ R_1[X] \rightarrow R_0[X]$, defined on paths $p$, i.e. generators of $R_1[X]$, by $\partial_0(p)=(x,y)$ where $x$ (respectively $y$) is the start (respectively the end) of path $p$. 
This map induces also maps, that we still denote by $\partial_0$, from $R_n[X]$, $n \geq 1$, to $R_0[X]$, in an obvious way. 

This is in general not an augmentation for the simplicial algebra structure, since this would mean we would have a map $\sigma_0: \ R_0[X] \rightarrow R_1[X]$ with $\partial_0 \circ \sigma_0=Id$, i.e. $\sigma_0$ would be a global section of the analogue of the "path fibration" $\chi$ in the directed space $X$, defined and studied in e.g. \cite{dirtop}: in this case $X$ would be dicontractible. In the sequel, we will instead see $\partial_0$ as an encoding of a bigrading of the simplicial algebra $R[X]$. 


\begin{remark}
\label{rem:functoriality}
We note the following: 
\begin{itemize}
\item In the proof above, we have shown, en passant, that the boundary and degeneracy operators of Definition \ref{def:boundaryoperators} give the sequence $({\cal T}_i)_{i\in \N}$ of categories, the structure of a simplicial object in Cat. 
\item 
Going from directed spaces $X$ to the simplicial object in $Alg$ $R[X]$ is not functorial in general, but it is for morphisms of directed spaces $f: \ X \rightarrow Y$ which are injective on points, as a consequence of Remark \ref{rem:notfunctorial}. 
\end{itemize}
\end{remark}

We can now define the homology algebra of a directed space and characterize its elements.



\section{Directed homology algebras}

\label{sec:diralg}

The construction of the homology from a simplicial object in $Alg$ is more involved than in the classical case, in order to control the quotients that have to be made, modulo boundaries. 

It is shown in \cite{VanderLinden} that the Moore normalization of a simplicial object (here in the semi-abelian category $Alg$) is a proper chain complex, meaning that the boundary map is proper, from which we can take the "classical" definition of homology, as kernel over image. 

Indeed, taking the "naive" definition of the chain complex out of the simplicial object in $Alg$, as well as the classical (unnormalized) differential $\partial$ out of it would give homology algebras to be $H_n=Ker \ \partial/I(\partial)$ where $I(\partial)$ is the two-sided ideal generated by $Im \ \partial$ ($\partial$ is indeed not proper in general). 
We would have little hope to get any classical homology long exact sequences from this definition, and it would even be difficult to characterize $H_n$ as a $R$-module. 

We will see that  the semi-abelian approach gives us something which is closely linked to the approach using natural systems \cite{Dubut}, see Proposition \ref{lem:charactHA}. 

For being as much self-contained as possible, we review this construction applied to the semi-abelian category of algebras: 

\begin{definition}[Moore normalization]
Let $S$ be a simplicial object in $Alg$, with boundary maps $d^n_i : \ S_n \rightarrow S_{n-1}$ and degeneracy maps $s^n_i: \ S_n \rightarrow S_{n+1}$, $i=0,\ldots,n$. Its Moore normalization is the following chain complex of algebras, $NS$: 
\begin{itemize}
\item it is in degree 
$n$
 the joint kernel $(NS)_n=\mathop{\cap}\limits_{i=0}^{n-1} ker \ d_i^n$
of all face maps except the $n$th face;
\item with differential given by the remaining $n$th-face map
$\partial^n=(-1)^n d_n^n: \ (NS)_n \rightarrow (NS)_{n-1}$
\end{itemize}
\end{definition}

As well-known \cite{VanderLinden} (and easy to check directly, this was already well-known for simplicial groups, see e.g. \cite{May}), $NS$ is a proper chain complex of algebras, from which we can take the homology of, $H_*(NS)$. For $X$ a directed space, we call $NR[X]$ the normalized simplicial set of the simplicial set $(R_{i+1}[X])_{i \in \N}$, and we can define: 

\begin{definition}
The homology algebra of a directed space $X$ is defined as the the homology $(HA_{i+1}(X))_{i\geq 0}$ in the semi-abelian category $Alg$ of the simplicial object defined in Theorem  \ref{lem:simplicialalgebra}, shifted by one, i.e. for all $i\geq 0$: 
$$
HA_{i+1}(X) = H_i(NR[X])
$$
i.e. is the classical homology of the normalized simplicial algebra $NR[X]$. 
For consistency purpose, we set 
$$HA_0(X)=R_0[X]$$
\end{definition}

\begin{remark}
Note that maps $\partial_0: \ R_n[X] \rightarrow R_0[X]$, $n\geq 0$, induces maps, still denoted by $\partial_0$, from $HA_n(X)$ to $R_0[X]$, i.e. all $HA_n(X)$ are bigraded by $R_0[X]$.  In what follows, we write $HA_n(X)(a,b)$ for the sub-$R$-module $\{ x\in HA_n(X) \ \mid \ \delta_0(x)=(a,b)\}$. 
\end{remark}



\section{Elementary properties of directed homology algebras}

\label{sec:elemprop}
In this section, we begin by noting that directed homology algebras form algebraic invariants of directed spaces (up to directed homeomorphism), hence are a tool for classifying directed spaces. We then go on by giving elementary algebraic properties of $HA_1$ and then of $HA_n$, $n > 1$: $HA_1$ is shown to be the algebra of classes of dipaths paths modulo homology with fixed extremities, and $HA_n$, $n>1$ is shown to be merely a module (with null multiplication) which is the direct sum of modules of $n$-traces modulo homology with fixed extremities. 
 
 \paragraph{Invariance under dihomeomorphism}
 
In what follows, we denote by $\mathcal{D}_I$ for the subcategory of directed spaces with injective directed maps. 

Let $p$ be an $i$-path in $X$: $p$ is a map from $\I \times \Delta_{i-1} \rightarrow X$. Then, for $f: \ X \rightarrow Y$ a morphism of directed spaces,  $f \circ p: \ \I\times \Delta_{i-1} \rightarrow Y$ is the $i$-path which is the direct image of $p$ by $f$. 
This indeed carries over to a mapping of $i$-traces and $\tilde{f}$ defines a map from $\mathcal{T}_i(X)$ to $\mathcal{T}_i(Y)$. By Remark \ref{rem:functoriality}, this defines a map from the algebra $R_i[X]$ to $R_i[Y]$, by injectivity of $f$. This map obviously commutes with the simplicial structure, thus carries over the homology construction of Section \ref{sec:diralg}. We thus have proved: 


\begin{lemma}
\label{lem:diralgfunct}
The directed homology algebra construction defines a functor for each $n \in \N$, $HA_n: \ \mathcal{D}_I \rightarrow Alg$. 
\end{lemma}


As a consequence we now have that directed homology algebras are algebraic invariants of directed spaces: 

\begin{proposition}
Let $X$ and $Y$ be two dihomeomorphic directed spaces. Then for all $n\in \N$, $HA_n(X)$ and $HA_n(Y)$ are isomorphic algebras.
\end{proposition}

\begin{proof}
Let $f: \ X \rightarrow Y$ with inverse $g: \ Y \rightarrow X$ be a dihomeomorphism between $X$ and $Y$. In particular $f$ and $g$ are injective so they induce algebra maps, for any $n\in \N$, $\tilde{f}: \ HA_n(X) \rightarrow HA_n(Y)$ and $\tilde{g}: \ HA_n(Y) \rightarrow HA_n(X)$ by Lemma \ref{lem:diralgfunct}. 
By functoriality (Lemma \ref{lem:diralgfunct} again), $f\circ g=g\circ f=Id$ implies $\tilde{f}\circ\tilde{g}=\tilde{g}\circ\tilde{f}=Id$. 
\end{proof}

\paragraph{Characterization of directed homology algebras}

Now we characterize these homology algebras. 

In the classical case of simplicial abelian groups \cite{May}, the homology of the normalized complex and of the unnormalized one are isomorphic. It is not in general the case for algebras, since the congruence generated by the image of the non-normalized boundary operator may differ significantly than the image of the unnormalized boundary operator. 

Still, it is the case the two coincide for $n=1$: 

\begin{lemma}
\label{lem:HA1}
$HA_1(X)$ is the coequalizer of $d_0$ and $d_1$ from $R_2[X]$ to $R_1[X]$. More explicitely, this is:
$$
HA_1(X) = R_1[X]/\sim 
$$ 
where $p$ and $q \in R_1[X]$ are such that $p \sim q$ if and only if there exists $z\in R_2[X]$ with: 
$$p=q+(d_0-d_1)(z)$$
\end{lemma}

\begin{proof}
We know by Corollary 2.2.13 of \cite{VanderLinden} that $HA_1(X)$ is the coequalizer of $d_0$ and $d_1$. 

The coequalizer of $d_0$ and $d_1$ is the quotient of $R_1[X]$ by the two-sided ideal generated by the image of $d_0-d_1$, see Section \ref{sec:backgroundalgebra}. The formula for the two-sided ideal generated by a set is an easy check:
$p \sim q$ if and only if there exists $u_j$, $v_j \in R_1[X]$ and $x_j \in Im \ (d_0-d_1)$, $j=1,\ldots,k$ with: 
$$p=q+\sum\limits_{j=1}^k u_j\times x_j\times v_j$$
As $x_j \in Im \ (d_0-d_1)$, we can write $x_j=(d_0-d_1)(y_j)$ for some $y_j \in R_2[X]$. Now, 
$$\begin{array}{rcl}
u_j \times (d_0-d_1)(y_j)\times v_j & = & u_j \times d_0(y_j)\times v_j - u_j\times d_1(y_j)\times v_j\\
& = & d_0(s_0(u_j) \times y_j\times s_0(v_j)) - d_1(s_0(u_j)\times y_j\times s_0(v_j)) \\
& = & (d_0-d_1)(s_0(u_j) \times y_j\times s_0(v_j))
\end{array}
$$
since by the simplicial identities and as $d_0$ and $d_1$ are algebra maps, Theorem \ref{lem:simplicialalgebra}. Therefore, $p \sim q$ if and only if there exists $z_j$ (equal to $s_0(u_j)\times y_j\times s_0(v_j)$) such that $p=q+\sum\limits_{j=1}^k(d_0-d_1)(z_j)=q+(d_0-d_1)(z)$ with $z=\sum_{j=1}^k z_j$. 
\end{proof}

We deduce that:

\begin{lemma}
For all $p \in T_1(X)$, $HA_1(X)(p_s,p_t)=HN_1(X)(p)$ as $R$-modules.
\end{lemma}

\begin{proof}
Take $p$ and $q\in R_1[X](p_s,p_t)$ such that $p \sim q$. There exists, by Lemma \ref{lem:HA1}, $z \in R_2[X]$ such that $p-q=(d_0-d_1)(z)$. 


This $z$ can be decomposed as a sum of $z^{a,b}\in R_2[X](a,b)$, $a$, $b \in X$ by Remark \ref{rem:decomppath}. Also by Theorem \ref{lem:simplicialalgebra}, 
$(d_0-d_1)(z^{a,b})\in R_1[X](a,b)$ so $p-q=(d_0-d_1)(z^{p_s,p_t})$. 

By Theorem \ref{lem:simplicialalgebra}, $R[X](a,b)$ is the free simplicial $R$-module on $ST(X)(a,b)$ so $p \sim q$ is equivalent to $p=q \ mod \ Im \ \partial$ where $\partial$ is the unnormalized differential of the free simplicial $R$-module on $ST(X)(p_s,p_t)$. 

Therefore $HA_1(X)$ which is composed, by Lemma \ref{lem:HA1}, of the classes of elements of $R_1[X]$ modulo $Im \ \partial$ is equal to the 0th homology group of $ST(X)(p_s,p_t)$, hence to $HN_1(X)(p)$, since $p$ is a 1-trace from $p_s$ to $p_t$, see Lemma  \ref{lem:naturalhomology}. 



\end{proof}

We will have the same result in all dimensions, at the module structure level, but we need first to understand the module structure of the directed homology algebras: 

\begin{lemma}
\label{lem:HARmod}
$HA_i(X)$ has as underlying $R$-module, the coproduct for $(a,b)\in R_0[X]$, of $HA_i(X)(a,b)$, which is equal to: 
$$Ker \ \partial^n_{\mid NR_n[X](a,b)} / Im \ \partial^{n+1}_{\mid NR_{n+1}[X](a,b)}$$ 
\end{lemma}

\begin{proof}
By Theorem \ref{lem:simplicialalgebra}, the simplicial structure on $R[X]$ restricts on $R[X](a,b)=(e^{i+1}_a R_{i+1}[X] e^{i+1}_b)_{i \geq 0}$ to the free simplicial $R$-module on $ST(X)(a,b)$, for all $a$, $b\in X$. Hence $Ker \ \partial^n_{\mid NR[X]}$ is the coproduct on all $a$, $b \in X$ of $Ker \ \partial^n_{\mid R_n[X](a,b)}$. Similarly for $Im \ \partial^{n+1}_{\mid NR[X]}$, which as a $R$-module is the coproduct of all the $Im \ \partial^{n+1}_{\mid R_{n+1}[X](a,b)}$. Now, $\partial^{n+1}$ is a proper map of algebras, 
hence the quotient of $Ker \ \partial^n$ by $Im \ \partial^{n+1}$ is, as a $R$-module, the coproduct of all the $Ker \ \partial^n_{\mid R_n[X](a,b)} / Im \ \partial^{n+1}_{\mid R_{n+1}[X](a,b)}$. 
\end{proof}

\begin{proposition}
\label{lem:charactHA}
Let $X$ be a directed space, $a, b \in X$, $R$ a ring. Then
$HA_n(X)$ is isomorphic, as a $R$-module, to the coproduct of $HN_n(X)(p)$ for $p$ varying on any set of paths, one for each pair $(a,b)$ in $X^2$. 
\end{proposition}


\begin{proof}
By Lemma \ref{lem:HARmod}
$HA_i(X)$ has as underlying $R$-module, the coproduct of all 
$HA_i(X)(a,b))=Ker \ \partial^n_{\mid NR_n[X](a,b)} / Im \ \partial^{n+1}_{\mid NR_{n+1}[X](a,b)}$. 

But $Ker \ \partial^n_{\mid NR_n[X](a,b)} / Im \ \partial^{n+1}_{\mid NR_{n+1}[X](a,b)}$ is the homology of the Moore normalization of $R[X](a,b)$ within the abelian category of $R$-modules, which is equal to the free simplicial $R$-module on $ST(X)(a,b)$ by Theorem \ref{lem:simplicialalgebra}. 

It is well-known that the homology of its Moore normalization in the category of simplicial $R$-modules is equal to the homology of its unnormalized chain complex, which is equal to the homology of $ST(X)(a,b)$, equal to $HN_n(X)(p)$ for any $p$ 1-trace from $a$ to $b$ by Lemma \ref{lem:naturalhomology}. 
%
%
%
\end{proof}

The algebra operation for higher directed homology algebras $HA_n(X)$, $n\geq 2$, is, as we are going to see, degenerate. 

By Theorem 2.4.6 of \cite{VanderLinden} we know the homology in dimension greater or equal than 1 in a semi-abelian category is an abelian object in that category. 
By a standard Eckmann-Hilton argument: an abelian object in the category of {\em unital associative algebras} has a binary operation $m$ which agrees with the module addition and with the algebra multiplication, hence the three are equal. In particular 1 is equal to 0 and the algebra is the null algebra. But this result only holds because of the units. 

In the category of non-unital algebras as we are considering here, we can still infer by a Eckmann-Hilton argument, that $m$ is equal to the module addition. We can further write, for all $x, y$ in the abelian object in $Alg$ we are considering: 
$$\begin{array}{rcl}
0 & = & m(0,0) \\
& = & m(0\times y,x\times 0) \\
& = & m(0,x)\times m(y,0) \\
& = & x \times y
\end{array}$$

Hence we proved that: 


\begin{lemma}
For all $n \geq 2$, $HA_n(X)$ has as algebra operation the zero multiplication: for all $x$, $y$, $x\times y=0$. 
\end{lemma}

So higher homology algebras should only be considered to be mere $R$-modules. 

\begin{remark}
There is no interesting idempotent in $HA_n(X)$, since the multiplication is equal to zero. But the semi-augmentation map $\partial_0$ serves as a bigrading of the module $HA_n[X]$. Indeed, 
$$HA_n(X)=\bigoplus\limits_{(a,b)\in R_0[X]} HA_n(X)({a,b})$$
\end{remark}

\paragraph{Relationship with natural systems with composition pairing}

In \cite{cameron}, the authors showed that the natural homology of \cite{Dubut} has the drawback of giving isomorphic natural systems for the directed space $X$ and the directed space $X^{op}$ with opposite directed structure, and proposed a way to break this isomorphism, by considering natural homology as a natural system plus an operation called "composition pairing", first introduced by T. Porter \cite{Porter2}. 



Let $\V$ be a category with finite products. Given a natural system $D$ on a category $\Cr$ with values in $\V$, recall from~\cite{Porter2} that a {composition pairing} associated to $D$ consists of two families of morphisms of~$\V$ 
\[
\big(\; \nu_{f,g} : D_f\times D_g \fl D_{fg}\;\big)_{f,g\in \Cr_1} \qquad \text{ and } \qquad \big(\; \nu_x : T \fl D_{1_x}\;\big)_{x\in \Cr_0}
\]
where $T$ is the terminal object in $\V$, the indexing $1$-cells $f$ and $g$ are composable, and such that the three following coherence conditions are satisfied:
\begin{enumerate} 
\item Naturality condition:  the diagram
\[
\xymatrix @C=4em @R=2.25em{
D_{f}\times\ D_{g}
  \ar[r] ^-{\nu_{f,g}}
  \ar[d] _-{D(u,1)\times D(1,v)}
& 
D_{fg} 
  \ar[d] ^-{D(u,v)}
\\
D_{uf}\times D_{gv} 
  \ar[r] _-{\nu_{uf,gv}}
&  
D_{ufgv}
}
\]
commutes in $\V$ for all $1$-cells $f,g,u,v$ in $\Cr_1$ such that the composites are defined.
\item The cocycle condition: the diagram 
\[
\xymatrix@C=6em@R=2.25em{
D_{f}\times D_{g}\times D_{h}
	\ar[r] ^-{\nu_{f,g}\times id_{D_{h}}}
	\ar[d] _-{id_{D_{f}}\times \nu_{g,h}}
&  
D_{fg}\times D_{h} 
	\ar[d] ^-{\nu_{fg,h}}
\\
D_{f}\times D_{gh} 
	\ar[r] _-{\nu_{f,gh}}
& 
D_{fgh}
}
\]
commutes for all $1$-cells $f,g$ and $h$ of $\Cr$ such that the composite $fgh$ is defined,
\item The unit conditions: the diagrams
\[
\xymatrix@C=4em@R=2.25em{
D_{f}  
& 
D_{f} \times D_{1_{y}} 
	\ar[l] _-{\nu_{f,1_{y}}}
\\
&   
D_{f}\times T
	\ar[u] _-{1_{D_{f}}\times\nu_{y}}
	\ar[ul] ^-{\cong}
}
\hskip1cm
\xymatrix@C=4em@R=2.5em{
D_{1_{x}}\times D_{f} 
	\ar[r] ^-{\nu_{1_{x},f}} 
& 
D_{f}
\\
T\times D_{f}   
	\ar[u] ^-{\nu_{x}\times 1_{D_{f}}}
	\ar[ur] _-{\cong}
& 
}    
\]
commute for every $1$-cell $\map{f}{x}{y}$ of $\Cr$.
\end{enumerate}

The category of natural systems on $\Cr$ with values in $\V$ which admit a composition pairing is the category whose objects are pairs $(D,\nu)$, with $D$ a natural system on $\Cr$ and $\nu$ a composition pairing associated to $D$. The morphisms are natural transformations $\alpha : D \fl D'$ {compatible with the composition pairings} $\nu$ and $\nu'$.

Now, 
let $(D,\nu)$ be a natural system with values in $R$-modules, on a small category $\cal C$, and $\nu$ is a composition pairing on $D$, then we claim that we can 
associate to $(D,\nu)$ the $R$-algebra $R[D,\nu]$: 
\begin{itemize}
    \item whose underlying module is the coproduct 
    $$
    \coprod\limits_{(a,b) \in Obj(\C), \ f \in \C(a,b)} D_f
    $$
    \item the external multiplication $a\times b$ for $a$, $b \in R[D,\nu]$ is defined on each component of the coproduct as: 
    $$
    a \times b =\left\{\begin{array}{ll}
    \nu_{f,g}(a,b) & \mbox{if $f$ and $g$ are composable and $a\in D_f$, $b \in D_g$} \\
    0 & \mbox{otherwise}
    \end{array}\right.
    $$
    \noindent and then, extended by bilinearity. 
\end{itemize}

This is immediate: the cocycle condition shows the associativity of $a\times b$. Distributivity over addition is ensured by naturality of $\nu$.



Conversely, we have the following, akin to the classical relationship \cite{assocalg} between $R$-algebras and representations of left or right modules over graphs. In what follows, we restrict to the case where $R$ is a field. 

Let $A$ be a basic and connected finite dimensional $R$-algebra and ${e_1,e_2,\ldots,e_n}$ be a complete set of primitive orthogonal idempotents of $A$. The (ordinary) quiver of $A$ \cite{assocalg}, denoted by $Q_A$, is defined as follows:
\begin{itemize}
\item The points of $Q_A$ are the numbers $1, 2,\ldots , n$, which are in bijective correspondence with the idempotents $e_1, e_2,\ldots , e_n$.
\item Given two points $a,b \in (Q_A)_0$, the arrows $\alpha : \ a \rightarrow b$ are in bijective correspondence with the vectors in a basis of the $R$-vector space $e_a(rad \ A/rad^2 \ A)e_b$, where $rad \ A$ is the Jacobson ideal of the algebra $A$ \cite{assocalg}. 
\end{itemize}

Construct now a natural system with composition pairing from the basic connected finite dimensional $R$-algebra $A$ as follows: 
\begin{itemize}
    \item Consider $\C$ the free category generated by quiver $Q_A$ defined above
        \item Construct a natural system as follows. We define a functor $F: {\cal F C}\rightarrow R-mod$ with: 
    \begin{itemize}
        \item Consider $f$ an object of ${\cal FC}$, 
        it corresponds to a directed path $(p_1,\ldots,$ $p_l)$ from some $a$ to some $b$ in $Q_A$, hence to an element of $e_a A e_b$. We set $$F(f)=e_a A e_b$$
        \item Consider $\langle u ,v \rangle$ a morphism from $f$ to $g$ in $\cal FC$. $f$ corresponds to a path from $a$ to $b$ in $Q_A$, $g$ to a path from $a'$ to $b'$, $u$ to a path from $a'$ to $a$ and finally $v$ to a path from $b$ to $b'$. We set 
        $$F(\langle u,v \rangle(x)=u\times x \times v
        $$
        \noindent $u$ seen as an element of $e_{a'} A e_a$, $v$ as an element of $e_b A e_{b'}$, for any $x\in e_a A e_b$. 
    \end{itemize}
    \item Construct a natural transformation $\nu: \ F_o \rightarrow {}_o F$ as follows: 
    $$
    \nu(u,v)=u \times v
    $$
    \noindent where $f$ is a path from $a$ to $b$ in $Q_A$, $g$ is a path from $b$ to $c$ in $Q_A$, $u \in e_a A e_b$, $v \in e_b A e_c$. 
\end{itemize}



The transformations from $R$-algebras to natural systems with composition pairing and vice-versa are inverse of one another. This would indicate a strong relation of the first natural homology natural system with composition pairing with our construction $HA_1$ (although we applied it only to directed spaces that give infinite-dimensional algebras), but shows the discrepancy between higher natural homology natural systems with composition pairing introduced in \cite{cameron} with $HA_n$, $n\geq 2$ that we constructed here. We believe that moving from $R$-algebras to $R$-algebroids for constructing homology theories for directed spaces should be the counterpart to the construction of \cite{cameron}. This will be done in another venue: algebroids cannot be treated in the semi-abelian framework contrarily to algebras as we developed. Hopefully, much is known about algebroids in the context of classical algebraic topology already \cite{Mosa,mitchell1970rings,mitchell1985separable,tbtkmath575197}. 






\section{Examples}

\label{sec:examples}

We give below simple illustrative examples. We consider d-spaces $\mid C \mid$ generated by finite precubical sets $C$, as described in Section \ref{sec:dspace}. We will suppose in what follows that $C$ is in fact a geometric precubical set, see e.g. \cite{Fajstrup2005DipathsAD}. We will describe the full algebra $HA_1(\mid C\mid )$, which is infinite dimensional in general, but also interesting finite dimensional subalgebras $\widetilde{HA}_1(\mid C\mid )$ (that we call, by definition, the first homology algebra of the precubical set $C$, $HA_1(C)$), which is extracted using the bigrading $\partial_0$. We only retain in $\widetilde{HA}_1(\mid C\mid )$ the elements of $HA_1(C)$ that have as bigradings pair of vertices in $C$: 
$$\widetilde{HA}_1(C)=\{ x \in HA(\mid C\mid ) \ \mid  \ \partial_0(x) \in (C_0)^2\} 
$$

We note that, by 
Theorem 4.1 of \cite{Fajstrup2005DipathsAD}, all dipaths of $\mid C\mid $ are dihomotopic, hence homologous in our theory, to cubical dipaths in $C$, i.e. dipaths which go through the geometric realization $\mid C_{\leq 1}\mid $ of the 1-skeleton $C_{\leq 1}$ of $C$ within $\mid C\mid $. This means that for describing $\widetilde{HA}_1(C)$ it is enough to describe first the path algebra of $C$. 

Let us define $R_2[C]$ to be the 2-dipaths algebra, whose generators are sequences of 2-cells $(A_1,\ldots,A_k)$ for some $k\geq 0$, $A_i \in C_2$, $i=1,\ldots,k$, and $d^1_0 d^1_0(A_i)=d^0_0d^0_0(A_{i+1})$ for $i=1,\ldots,k-1$, and whose (external) multiplication is: 
$$
(A_1,\ldots,A_k)*(B_1,\ldots,B_l)=\left\{\begin{array}{ll}
(A_1,\ldots,A_k,B_1,\ldots,B_l) & \mbox{if $d^1_0 d^1_0(A_k)=d^0_0d^0_0(B_{1})$} \\
0 & \mbox{otherwise}
\end{array}\right.
$$

By 
Theorem 5.2 of \cite{Fajstrup2005DipathsAD}, two dipaths in $\mid C_{\leq 1}\mid $ are dihomotopic if and only if they are dihomotopic within the geometric realization $\mid C_{\leq 2}\mid $ of the 2-skeleton $C_{\leq 2}$ of $C$, within $\mid C\mid $. 
It is easily seen now by
Lemma \ref{lem:HA1} that $\widetilde{HA}_1(C)$ is the quotient: 
$$
\widetilde{HA}_1(C)= R_1[C]/\sim 
$$ 
where $p$ and $q \in R_1[C]$ are such that $p \sim q$ if and only if there exists $z\in R_2[C]$ with: 
$$p=q+(\delta_0-\delta_1)(z)$$

\begin{example}[Directed $S^1$: $dS^1$]
\label{ex:dirS12}
We consider the directed realization $\mid dS^1\mid $ of the precubical set (a graph here) $dS^1$ of Example \ref{ex:dirS1}. 

Let $x$ and $y$ be two distinct points on $\mid dS^1\mid $. The generators of $HA_1(\mid dS^1\mid )$ are easily seen to be the unique class of directed paths from $x$ to $y$ modulo increasing reparameterization, that we denote by $[x,y]$, and the unique class of directed paths from $y$ to $x$ modulo reparameterization that we denote by $[y,x]$. We denote by $s_x$ a representative dipath starting (and finishing) at $x$ a point of $\mid dS^1\mid $ of the unique generator of $\pi_1(\mid dS^1\mid )$ (we take this generator to agree with the directed structure on $\mid dS^1\mid $). 

The constant paths, idempotents of the algebra $HA_1(\mid dS^1\mid )$ are $[x,x]$. The algebra multiplication is defined as follows: 
$$\begin{array}{lcl}
{[}x,x]*[y,z] & = & \left\{\begin{array}{ll}
0 & \mbox{if $x \neq y$} \\
{[}y,z] & \mbox{otherwise}
\end{array}\right.\\
{[}y,z]*[x,x] & = & \left\{\begin{array}{ll}
0 & \mbox{if $x \neq z$} \\
{[}y,z] & \mbox{otherwise}
\end{array}\right.\\
{[}x,y]*[y,z] & = & \left\{\begin{array}{ll}
s_x * [x,z] & \mbox{if $z \in [x,y]$} \\
{[}x,z] & \mbox{otherwise}
\end{array}\right.
\end{array}$$
It is easy to see that 
$$\widetilde{HA}_1(dS^1)=R[s_u]$$
\noindent the algebra of univariate polynomials. This is because the only generator left to consider is $s_u$, where $u$ is the unique 0-cell of precubical set $S^1$, and because of the algebra laws on generators given above. 
\end{example}

\begin{example}[Filled-in and empty squares]
\label{ex:filledinsquare}
We consider now the following cubical complex $X$: 
\[\begin{tikzcd}
  4 \arrow[r,"a"] \arrow[d,"b"] \arrow[dr,phantom,"C"]
    & 2 \arrow[d,"c"] \\
  3 \arrow[r,"d"]
& 1 \end{tikzcd}
\]
\noindent with a two cell $C$ filling in the corresponding hole. 

The geometric realization of $X$ is dihomeomorphic to $I^2$. 
The algebra $HA_1(\mid X\mid )$ is easily seen to be generated by elements $[x,y]$ with $x$ and $y$ elements of $I^2$, such that $x\leq y$ in the componentwise ordering. The composition is generated by:
$$[x,y]\times [y,z]=[x,z]$$

Now, consider $Y$ to be the 1-skeleton of $X$. We see $\mid Y\mid $ as the boundary of the unit square $I^2$. The generators of $HA_1(\mid Y\mid )$ are now $[x,y]$ as before, when $x\leq y$ as before, except when $x$ is vertex $4$ and $y$ is vertex $1$, in which case we have two generators $[4,1]_3$ corresponding to the class of paths from 4 to 1 going through 3 and $[4,1]_2$ corresponding to the class of paths from 4 to 1 going through 2. The composition is inherited by the one on $HA_1(\mid X\mid )$. 

Now we are looking at $\widetilde{HA}_1(X)$ and $\widetilde{HA}_1(Y). 
$
We had the following path algebra (see Example \ref{ex:emptysquare}) for the empty square, i.e. for $R_1[X]$: 
$$
\begin{pmatrix}
R & 0 & 0 & 0 \\
R & R & 0 & 0 \\
R & 0 & R & 0 \\
R^2 & R & R & R 
\end{pmatrix}
$$
For the empty square $Y$, there is nothing to quotient from there, and $\widetilde{HA}_1(Y)=R_1[X]$. For the filled-in square $X$, 
we quotient it by the (admissible) sub-$R_1[X]$-module generated by $ac-bd$, which is $Im \ \partial$ for $\partial: \ R_2[X] \rightarrow R_1[X]$. $Im \ \partial$ can be seen also as the two-sided ideal of the algebra $R_1[X]$ generated by $ac-bd$. The resulting quotient of $R_1[X]$-module can then be seen as a $R$-algebra, which is the following matrix algebra: 
$$
\begin{pmatrix}
R & 0 & 0 & 0 \\
R & R & 0 & 0 \\
R & 0 & R & 0 \\
R & R & R & R 
\end{pmatrix}
$$

\end{example}



\begin{example}[Two holes on the antidiagonal of a square]
\label{ex:twoholes}
We consider the following two precubical sets and describe $\widetilde{HA}_1$ of both: 
\begin{center}
    \begin{minipage}{6cm}
    \[\begin{tikzcd}
  9 \arrow[r,"i"] \arrow[d,"k"] \arrow[dr,phantom,"C"] & 8 \arrow[d,"h"] \arrow[r,"j"] & 7 \arrow[d,"l"]\\
  6 \arrow[r,"d"] \arrow[d,"c"] & 5 \arrow[d,"e"] \arrow[dr,phantom,"D"] \arrow[r,"f"] & 4 \arrow[d,"g"] \\
  3 \arrow[r,"a"] & 2 \arrow[r,"b"] & 1
\end{tikzcd}
\]
\end{minipage}
    \begin{minipage}{6cm}
    \[\begin{tikzcd}
  9 \arrow[r,"i"] \arrow[d,"k"]  & 8 \arrow[d,"h"] \arrow[r,"j"] \arrow[dr,phantom,"E"] & 7 \arrow[d,"l"]\\
  6 \arrow[r,"d"] \arrow[d,"c"] \arrow[dr,phantom,"F"] & 5 \arrow[d,"e"]  \arrow[r,"f"] & 4 \arrow[d,"g"] \\
  3 \arrow[r,"a"] & 2 \arrow[r,"b"] & 1
\end{tikzcd}
\]
\end{minipage}
\end{center}
For both complexes, we have the following path algebra $R_1[C]$: 
$$
\begin{pmatrix}
R & 0 & 0 & 0 & 0 & 0 & 0 & 0 & 0 \\
R & R & 0 & 0 & 0 & 0 & 0 & 0 & 0 \\
R & R & R & 0 & 0 & 0 & 0 & 0 & 0 \\
R & 0 & 0 & R & 0 & 0 & 0 & 0 & 0 \\
R^2 & R & 0 & R & R & 0 & 0 & 0 & 0 \\
R^3 & R^2 & R & R & R & R & 0 & 0 & 0 \\
R & 0 & 0 & R & 0 & 0 & R & 0 & 0 \\
R^3 & R & 0 & R^2 & R & 0 & R & R & 0 \\
R^6 & R^3 & R & R^3 & R^2 & R & R & R & R 
\end{pmatrix}
$$
For the left cubical complex, we quotient it by $Im \ \partial$ which is the sub-$R_1[X]$-module generated by $ih-kd$ and $gf-be$. For the right cubical complex, we quotient it by $Im \ \partial$ which is the module generated by $jl-hf$ and $de-ca$, giving respectively the matrix algebras, by the same argument as above: 
$$
\begin{pmatrix}
R & 0 & 0 & 0 & 0 & 0 & 0 & 0 & 0 \\
R & R & 0 & 0 & 0 & 0 & 0 & 0 & 0 \\
R & R & R & 0 & 0 & 0 & 0 & 0 & 0 \\
R & 0 & 0 & R & 0 & 0 & 0 & 0 & 0 \\
R & R & 0 & R & R & 0 & 0 & 0 & 0 \\
R^3 & R^2 & R & R & R & R & 0 & 0 & 0 \\
R & 0 & 0 & R & 0 & 0 & R & 0 & 0 \\
R^3 & R & 0 & R^2 & R & 0 & R & R & 0 \\
R^6 & R^3 & R & R^3 & R & R & R & R & R 
\end{pmatrix}
$$
\noindent and
$$
\begin{pmatrix}
R & 0 & 0 & 0 & 0 & 0 & 0 & 0 & 0 \\
R & R & 0 & 0 & 0 & 0 & 0 & 0 & 0 \\
R & R & R & 0 & 0 & 0 & 0 & 0 & 0 \\
R & 0 & 0 & R & 0 & 0 & 0 & 0 & 0 \\
R^2 & R & 0 & R & R & 0 & 0 & 0 & 0 \\
R^3 & R & R & R & R & R & 0 & 0 & 0 \\
R & 0 & 0 & R & 0 & 0 & R & 0 & 0 \\
R^3 & R & 0 & R & R & 0 & R & R & 0 \\
R^6 & R^3 & R & R^3 & R^2 & R & R & R & R 
\end{pmatrix}
$$
which are indeed non-isomorphic algebras. 
\end{example}

\begin{example}[Empty cube]
\label{ex:emptycube}
In the following example, we consider the cubical complex whose underlying quiver is: 
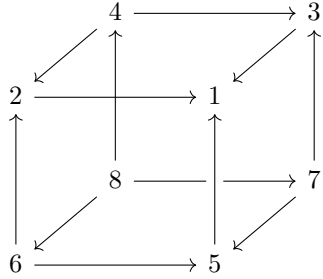
\[\begin{tikzcd} 
    &  4\arrow{rr} \arrow{dl} & &   3  \arrow{dl} \\
    2 \arrow[crossing over]{rr} & & 1 \\
      & 8 \arrow{rr} \arrow{uu} \arrow{dl} & &  7  \arrow{dl} \arrow[uu] \\
    6 \arrow{rr}\arrow{uu} && 5 \arrow[uu,crossing over]
 \end{tikzcd}\]
 \noindent and whose 6 faces $(8,6,5,7)$, $(6, 5, 2, 1)$ etc. are filled in by 2-cells. 
 
 By Lemma \ref{lem:charactHA}, $HA_2(X)$ is the module $HA_2(X)(8,1)$ (the other ones are equal to 0), which has just one generator. 
\end{example}



\section{Exact sequences of directed homology algebras}

\label{sec:exactseq}

As directed homology is defined in terms of a simplicial object in a semi-abelian category, we know \cite{VanderLinden} that any short exact sequence in the category of simplicial objects of algebras: 

\[
\begin{tikzcd}
    0\arrow{r} & A\arrow{r}{f} & B\arrow{r}{g} & C\arrow{r} & 0 \\
\end{tikzcd}
\]
\noindent gives rise to a corresponding long exact sequence between homology algebras: 

 \tikzset{
  curarrow/.style={
  rounded corners=8pt,
  execute at begin to={every node/.style={fill=red}},
    to path={-- ([xshift=-50pt]\tikztostart.center)
    |- (#1) node[fill=white] {$\scriptstyle d_*$}
    -| ([xshift=50pt]\tikztotarget.center)
    -- (\tikztotarget)}
    }
}

\begin{center}
    \begin{tikzcd}[arrow style=math font,cells={nodes={text height=2ex,text depth=0.75ex}}]
       \cdots & H_{q-1}(B) \arrow[l] \arrow[draw=none]{d}[name=Y, shape=coordinate]{} & \arrow[l] H_{q-1}(A) \\
       H_{q}(C) \arrow[curarrow=Y]{urr}{} & H_{q}(B) \arrow[l] \arrow[draw=none]{d}[name=Z,shape=coordinate]{} & H_{q}(A) \arrow[l] \\
       H_{q+1}(C) \arrow[curarrow=Z]{urr}{} & H_{q+1}(B) \arrow[l] & \cdots \arrow[l]
   \end{tikzcd}
\end{center}

A natural question is whether we have a Mayer-Vietoris long exact sequence in our context. 

Consider a directed space, union of directed spaces $X_1$ and $X_2$. Consider the coproduct $R_i[X_1]\oplus R_i[X_2]$ of the algebras $R_i[X_1]$ and $R_i[X_2]$ of $i$-traces. Its elements are linear combinations of elements of the form: 
$$x_1 \otimes x_2 \otimes \ldots \otimes x_n$$
\noindent where $x_i \in X_{\epsilon_i}$, $\epsilon_i \in \{1,2\}$, see Section \ref{sec:catalg}. As we have injective maps of directed spaces from $X_1$ to $X$ and from $X_2$ to $X$, by Remark \ref{rem:functoriality}, we have algebra maps from all $R_i[X_j]$ to $R_i[X]$. By the universal property of coproducts, this implies that there is a unique algebra map 
$$h: \ R_i[X_1]\oplus R_i[X_2] \rightarrow R_i[X]$$
\noindent which is easily seen to be such that:
$$h(x_1 \otimes x_2 \otimes \ldots \otimes x_n)=x_1\times x_2\times \ldots x_n$$
\noindent where the algebra multiplication is the one in $X$. 

It is immediate to see that, as the boundary and degeneracy operators of Theorem \ref{lem:simplicialalgebra} are algebra maps, $R[X_1]$ and $R[X_2]$ admit a coproduct as simplicial algebras, with $(R[X_1]\oplus R[X_2])_i=R_i[X_1]\oplus R_i[X_2]$ for all $i \in \N$, and: 
$$\begin{array}{lcl}
d_i(x_1 \otimes x_2 \otimes \ldots \otimes x_n) & = & d_i(x_1)\otimes d_i(x_2)\otimes \ldots \otimes d_i(x_n) \\
\sigma_i(x_1 \otimes x_2 \otimes \ldots \otimes x_n) & = & \sigma_i(x_1)\otimes \sigma_i(x_2)\otimes \ldots \otimes \sigma_i(x_n) 
\end{array}$$
\noindent and map $h$ above is a map of simplicial algebras. 

\paragraph{Disjoint unions of directed spaces}
When $X_1$ and $X_2$ have empty intersection, i.e. when $X$ is the disjoint union of $X_1$ with $X_2$, $h$ is the zero map on all tensors $R_i[X_{\epsilon_1}]\otimes \ldots \otimes R_i[X_{\epsilon_n}]$, with $i\geq 0$ and $n\geq 2$, $\epsilon_j \neq \epsilon_{j+1}$ for all $j=1,\ldots, n-1$. Function $h$ is only non-zero on 
elements $x_1 \in R_i[X_1]$ (respectively $x_2 \in R_i[X_2]$) where it is the identity map. Function $h$ is therefore surjective from $R_i[X_1]\oplus R_i[X_2]$ to $R_i[X_1 \cup X_2]$. Thus, we have the following short exact sequence of algebras, which is actually a short sequence of simplicial algebras, since kernels of algebras defined in Section \ref{sec:catalg} induce kernel of simplicial algebras:
\[
\begin{tikzcd}
    0\arrow{r} & Ker \ h \arrow{r}{ker \ h} & R[X_1]\oplus R[X_2] \arrow{r}{h} & R[X] \arrow{r} & 0 \\
\end{tikzcd}
\]
\noindent Therefore we get the following long exact sequence of homology algebras: 

\begin{center}
    \begin{tikzcd}[arrow style=math font,cells={nodes={text height=2ex,text depth=0.75ex}}]
       \cdots & H_{q-1}(R[X_1]\oplus R[X_2]) \arrow[l] \arrow[draw=none]{d}[name=Y, shape=coordinate]{} & \arrow[l] HA_{q-1}(Ker \ h) \\
       HA_{q}(X) \arrow[curarrow=Y]{urr}{} & H_{q}(R[X_1]\oplus R[X_2]) \arrow[l] \arrow[draw=none]{d}[name=Z,shape=coordinate]{} & HA_{q}(Ker \ h) \arrow[l] \\
       HA_{q+1}(X) \arrow[curarrow=Z]{urr}{} & H_{q+1}(R[X_1]\oplus R[X_2]) \arrow[l] & \cdots \arrow[l]
   \end{tikzcd}
\end{center}

Now, it is straightforward to see that $H_i(R[X_1]\oplus R[X_2])=HA_i[X_1]\oplus HA_i[X_2]$. This goes as in the abelian case: the Moore normalization of $R[X_1] \oplus R[X_2]$ is generated by $x_1 \otimes \ldots \otimes x_n$ where $x_i$ is in the Moore normalization of $R[X_{\epsilon_i}]$, for all $i=1,\ldots,n$. This is because the boundary operators $d_i$ act on $x_1\otimes \ldots \otimes x_n$ as $d_i(x_1)\otimes \ldots \otimes d_i(x_n)$. For the same reason, the image of $\partial$ is generated by $y_1\otimes \ldots \otimes y_n$ where $y_i$ is in the image of $\partial$ in $R[X_{\epsilon_i}]$. Hence we have the following Mayer-Vietoris like long exact sequence: 


 \begin{center}
    \begin{tikzcd}[arrow style=math font,cells={nodes={text height=2ex,text depth=0.75ex}}]
       \cdots & HA_{q-1}(X_1)\oplus HA_{q-1}(X_2) \arrow[l] \arrow[draw=none]{d}[name=Y, shape=coordinate]{} & \arrow[l] HA_{q-1}(Ker \ h) \\
       HA_{q}(X) \arrow[curarrow=Y]{urr}{} & HA_{q}(X_1)\oplus HA_q(X_2) \arrow[l] \arrow[draw=none]{d}[name=Z,shape=coordinate]{} & HA_{q}(Ker \ h) \arrow[l] \\
       HA_{q+1}(X) \arrow[curarrow=Z]{urr}{} & HA_{q+1}(X_1)\oplus HA_q(X_2) \arrow[l] & \cdots \arrow[l]
   \end{tikzcd}
\end{center}
\noindent with $q\geq 1$ (the rest being the 0 algebra). 

\paragraph{Towards a Mayer-Vietoris like long-exact sequence}

\label{sec:mayer}
When $X$ is not the disjoint union of $X_1$ and $X_2$, i.e. $X_1 \cap X_2\neq \emptyset$, 
we have the short exact sequence of simplicial algebras: 
\[
\begin{tikzcd}
    0\arrow{r} & Ker \ h \arrow{r}{ker \ h} & R[X_1]\oplus R[X_2] \arrow{r}{h} & R[X] \\
\end{tikzcd}
\]
\noindent but there is not reason a priori that map $h$ is surjective. It is though, when restricting $h$ to a map from  $R_1[X_1]\oplus R_1[X_2]$ to $R_1[X]$. 

The proof goes as for the classical van Kampen theorem on groupoids \cite{vankampen}, or for fundamental categories \cite{goubaultvankampen,grandisbook}. Let $p$ be a 1-trace of $R_1[X]$. There is a Lebesgue covering of $[0,1]$ such that $p$ is the concatenation of finitely many of its subpaths $p_1,\ldots,p_k$, with $p_i \in X_{\epsilon_i}$, $i=1,\ldots, k$, $\epsilon_i=1$ or 2. Then $h(p_1\otimes \ldots \otimes p_k)=p$, and $p$ is in the image of $h$. 

For higher traces in $X$, this is not true in general: for instance if $X_1$ (with edges $a_1$, $b_1$, $c_1$, $d_1$ and 2-cell $C_1$) and $X_2$ (with edges $a_2$, $b_2$, $c_2$, $d_2$ and 2-cell $C_2$) are two copies of the unit square of the Example \ref{ex:filledinsquare}, glued together along $c_1$ and $d_2$, the 2-trace which is the equivalence class under reparametrization of any directed homeomorphism from $\I \times I$ to $X$ is not the concatenation, in $R_2[X]$ of any 2-trace of $X_1$ with some 2-trace of $X_2$. 

There are still simple cases in which this holds and we have a long-exact sequence as for the disjoint union above. For instance all $i$-traces in the space of Example \ref{ex:twoholes} are concatenations of $i$-traces of the space geometric realization of the sub-precubical sets made up of $2, 3, 5, 6, 8, 9, a, c, e, d, k, h, i, C$, with $i$-traces of the space geometric realization of the sub-precubical sets made up of $2, 1, 5, 4, 8, 7, b, e, g, h, f, l, j, D$. 
Indeed, the union of directed spaces with a discrete $X_1\cap X_2$ is a particular case where the long-exact sequence above holds.  
We give below a simple example: 

\begin{example}
Consider again the directed circle of Example \ref{ex:dirS1} and \ref{ex:dirS12}. We consider the directed circle to be the union of two half-circles $u$ and $v$: 
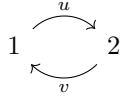
\[
\begin{tikzcd}
1 \arrow[r, out=45,in=135, "u"] & 2 \arrow[l, out=-135, in=-45, "v"]
\end{tikzcd}
\]

Using the same notations as in Example \ref{ex:dirS12}, we note by $[x,x']$ the unique 1-trace from $x$ to $x'$ on $u$ when $x \leq x'$ in the corresponding directed structure, and $[y,y']$ the unique 1-trace from $y$ to $y'$ on $v$, when $y \leq y'$. The algebra operations in $R_1[u]$ (and hence $HA_1[u]$) is similar to the one of Example \ref{ex:dirS12}: 
$$[x_1,x'_1]\times [x_2,x'_2] = \left\{\begin{array}{ll}
[x_1,x'_2] & \mbox{if $x'_1=x_2$} \\
0 & \mbox{otherwise}
\end{array}\right.
$$
\noindent and similarly for $v$. Thus $\widetilde{HA}_1[u]$ is the upper triangular matrix algebra: 
$$\widetilde{HA}_1[u]=\left(\begin{array}{cc}
R & R \\
0 & R
\end{array}\right)
$$
\noindent and $\widetilde{HA}_1[v]$ is the lower triangular matrix algebra (still indexing the lines and columns by vertices 1 and 2): 
$$\widetilde{HA}_1[v]=\left(\begin{array}{cc}
R & 0 \\
R & R
\end{array}\right)
$$
Now, $Ker \ h$ and $HA_1(Ker \ h)$ are easily seen to be the ideal of $R_1[u]\oplus R_2[v]$ generated by: 
$$\begin{array}{rcll}
[x,x']\otimes[y,y'] & = & 0 & \mbox{if $x'\neq 2$ or $y\neq 2$} \\ 
{[}y,y']\otimes[x,x'] & = & 0 & \mbox{if $y'\neq 1$ or $x\neq 1$} \\
\end{array}$$
\end{example}

Still, we believe that the long exact sequence above still hold in more general cases. The idea is that for some class of precubical sets, the homotopy type of tame paths in the sense of \cite{Ziemanski2} is the same as the homotopy type of all paths, and in some ways, $h$ is essentially surjective. 

\section{Conclusion and future work}


We provided in this paper a semi-abelian approach to directed homology. 
Many theoretical questions arise from this work. Can we develop a Kunneth-like theorem and long-exact sequence in relative homology (based on e.g. \cite{Goedecke}), can we develop a Mayer-Vietoris like long exact sequence in general from the discussion of Section \ref{sec:mayer}, using e.g. the methods from \cite{everaert}? What is the relationship between the Quillen-Andr\'e homologies \cite{Quillenrings} of the algebras $R_i[X]$, $i\geq 1$, for some directed space $X$, with the homology algebras we introduced here, if any? 

We touched upon the characterization of directed homology algebras for precubical sets. If fully developed with suitable relationships with the directed homology algebras of there directed geometric realizations \cite{thebook}, our theory is amenable to practical computations, using such mathematical software as GAP \cite{GAP4}. This is an exciting prospect, since most directed homology theories, such as natural homology, are not amenable to tractable computations. 

Finally, we hope to make clear the relationship between persistence homology and our approach in a subsequent paper. Indeed, one can consider a form of persistence homology on directed spaces $X$ by looking at some particular modules over the algebra $R_1[X]$.

\section{Declarations}

\paragraph{Funding and/or Conflicts of interests/Competing interests} 
The author has no competing interests to declare that are relevant to the content of this article.

\end{document}